\documentclass[UTF-8,reqno]{amsart}
\usepackage{enumerate}
\usepackage{amssymb,url,color, booktabs}

\usepackage{mathrsfs}

\usepackage{fancyhdr}
\usepackage[nobysame]{amsrefs}
\BibSpec{article}{%
+{}{\PrintAuthors} {author}
+{,}{ \textrm} {title}
+{.}{ \textit} {journal}
+{,}{ \textbf} {volume}
+{}{ \parenthesize} {date}
+{,}{ } {pages}
+{.}{ arXiv:} {eprint}
+{.}{} {transition}
}
\BibSpec{book}{%
+{}{\PrintAuthors} {author}
+{,}{ \textit} {title}
+{.}{ \textrm} {series} 
+{,}{ Vol.} {volume}
+{.}{ } {publisher}
+{,}{ } {date}
+{.}{} {transition}
}
\usepackage{color}
\usepackage[colorlinks=true]{hyperref}
\hypersetup{
    linkcolor=blue,          
    citecolor=red,        
    filecolor=blue,      
    urlcolor=cyan
}

\numberwithin{equation}{section}

\newcommand{\be}{\begin{eqnarray}}
\newcommand{\ee}{\end{eqnarray}}
\newcommand{\ce}{\begin{eqnarray*}}
\newcommand{\de}{\end{eqnarray*}}
\newtheorem{theorem}{Theorem}[section]
\newtheorem{lemma}[theorem]{Lemma}
\newtheorem{remark}[theorem]{Remark}
\newtheorem{definition}[theorem]{Definition}
\newtheorem{proposition}[theorem]{Proposition}
\newtheorem{Examples}[theorem]{Example}
\newtheorem{corollary}[theorem]{Corollary}
\newenvironment{proof of theorem 1.1}{{\it Proof of Theorem 1.1}.}{{\hfill $\square$\hskip - \parfillskip}}

\def\tiaojian{smooth, closed and uniformly convex }

\def\a{\alpha}

\def\p{\partial}

\def\g{\gamma}
\def\l{\lambda}

\def\[{{\Big[}}
\def\]{{\Big]}}
\def\<{{\langle}}
\def\>{{\rangle}}
\def\({{\Big(}}
\def\){{\Big)}}

\def\bx{{\mathbf{x}}}

\def\min{{\mathord{{\rm min}}}}

\def\={&\!\!=\!\!&}

\def\mR{{\mathbb R}}
\def\mS{{\mathbb S}}

\def\1{{\mathbf{1}}}

\def\geq{\geqslant}
\def\leq{\leqslant}
\def\ge{\geqslant}
\def\le{\leqslant}

\def\a{\alpha}

\def\p{\partial}

\def\g{\gamma}
\def\l{\lambda}

\def\[{{\Big[}}
\def\]{{\Big]}}
\def\<{{\langle}}
\def\>{{\rangle}}
\def\({{\Big(}}
\def\){{\Big)}}

\def\bx{{\mathbf{x}}}

\def\min{{\mathord{{\rm min}}}}

\def\={&\!\!=\!\!&}
\def\bt{\begin{theorem}}
\def\et{\end{theorem}}
\def\bl{\begin{lemma}}
\def\el{\end{lemma}}
\def\br{\begin{remark}}
\def\er{\end{remark}}
\def\bx{\begin{Examples}}
\def\ex{\end{Examples}}
\def\bd{\begin{definition}}
\def\ed{\end{definition}}
\def\bp{\begin{proposition}}
\def\ep{\end{proposition}}
\def\bc{\begin{corollary}}
\def\ec{\end{corollary}}

\def\geq{\geqslant}
\def\leq{\leqslant}
\def\ge{\geqslant}
\def\le{\leqslant}

 \def\R{\mathbb R}
 \def\R{\mathbb R}

\def\<{\langle} \def\>{\rangle}

\def\bpf{\begin{proof}}
\def\epf{\end{proof}}

\allowdisplaybreaks

\begin{document}
	
\title{A Class Of Curvature Flows Expanded By Support Function And Curvature Function}
\author{Shanwei Ding and Guanghan Li}


\address{School of Mathematics and Statistics, Wuhan University, Wuhan 430072, China.
}

\begin{abstract}
In this paper, we consider an expanding flow of closed, smooth, uniformly convex hypersurface in Euclidean $\mathbb{R}^{n+1}$ with speed $u^\alpha f^\beta$ ($\alpha, \beta\in\mathbb{R}^1$), where $u$ is support function of the hypersurface, $f$ is a smooth, symmetric, homogenous of degree one, positive function of the principal curvature radii of the hypersurface. If $\alpha \le 0<\beta\le 1-\alpha$, we prove that the flow has a unique smooth and uniformly convex solution for all time, and converges smoothly after normalization, to a round sphere centered at the origin.
\end{abstract}

\maketitle
\setcounter{tocdepth}{2}
\tableofcontents

\section{Introduction}
Flows of convex hypersurfaces in $\mathbb{R}^{n+1}$ by a class of speed functions which are homogenous and symmetric in principal curvature have been extensively studied in the past four decades. Well-known examples include the mean curvature flow \cite{HG}, and the Gauss curvature flow \cite{BS,FWJ}. In \cite{HG} Huisken showed that the flow has a unique smooth solution and the hypersurface converges to a round sphere if the initial hypersurface is closed and convex. Later, a range of flows with the speed of homogenous of degree one in principal curvatures were established,  see \cite{B0,B1,CB1,CB2} and references therein.

For the problem on the existence of the prescribed polynomial of the principal curvature radii of the hypersurface, Urbas \cite{UJ}, Chow and Tsai \cite{CB3}, Gerhardt \cite{GC}, Xia \cite{XC} studied the convergence for the flow with the speed of $f(\lambda_1,...,\lambda_n)$, where $f$ is symmetric polynomial of the principal curvature radii $\lambda_i$ of the hypersurface. In \cite{UJ}, Urbas considered the case of expanding convex hypersurfaces, and he proved that a smooth solution of the equation $\frac{\p X}{\p t}=f\nu$ exists for all time and that $M_t$ becomes spherical as $t\to\infty$, where $f$ is homogenous of degree one and satisfies some standard conditions, and $\nu$ is the outer unit normal vector field to $M_t$. Flow with speed depending not only on the curvatures has recently begun to be considered. For example, flows that deform hypersurfaces by their curvature and support function were studied in \cite{IM,SWM}. As a natural extension, in this paper, we consider the expanding flows of the convex hypersurfaces at the speed of $u^\alpha f^\beta$ with $\alpha, \beta\in\mathbb{R}^1$, where $u$ is the support function, and $f$ is a smooth, symmetric, homogenous of degree one, positive function of the principal curvature radii of the hypersurface. When $f=\sigma_k$, the flow has been studied by Sheng and Yi in \cite{SWM}.

Let $M_0$ be a closed, smooth, uniformly convex hypersurface in $\mathbb{R}^{n+1}$ ($n\geq2$), and $M_0$ encloses the origin. In this paper, we study the following expanding flow
\begin{equation}
\begin{cases}
&\frac{\partial X}{\partial t}(x,t)=<X,\nu>^\alpha f^\beta(x,t) \nu(x,t),\\
&X(\cdot,0)=X_0,
\end{cases}
\end{equation}
where $f(x,t)$ is a suitable curvature function of the hypersurface $M_t$ parametrized by $X(\cdot,t):\mS^n\to \mR^{n+1}$, $\beta>0$ and $\nu(\cdot,t)$ is the outer unit normal vector field to $M_t$.

To formulate our results, we shall suppose that the curvature function $f$ can be expressed as $f(\cdot,t)=f (\lambda_1,...,\lambda_n)$, where $\lambda_1,...,\lambda_n$ are the principal radii of curvature of the hypersurface $M_t$, and $f\in C^\infty(\Gamma^+)$ is a positive, symmetric function on the positive cone $\Gamma^+$=\{($\lambda_1,...,\lambda_n)\in \mR^n:\lambda_i>0$ for all $i$\}. The function $f$ is assumed to satisfy the following conditions:
\begin{gather}
\text{$f$ is homogenous of degree one on }  \Gamma^+,\\
\frac{\partial f}{\partial\lambda_i}>0\text{ on }\Gamma^+.
\end{gather}

In this paper, we prove the following
\begin{theorem}
Let $M_0$ be a closed, smooth, uniformly convex hypersurface in $\R^{n+1}$, $n\geq2$, enclosing the origin. Assume $\alpha, \beta\in \mR^1$ satisfying $\alpha \le 0<\beta\le 1-\alpha$. If $f\in C^\infty(\Gamma^+)$ is a positive, symmetric function on the positive cone satisfying (1.2) and (1.3), and satisfies the following conditions:
\begin{equation*}
\begin{split}
(\romannumeral1) \;& \text{the function $f$ is inverse concave, i.e. the dual function $f_*$ defined by}\\ &f_*(\l_1,...\l_n)=\frac{1}{f(1/\l_1,...,1/\l_n)}\text{ is concave on } \Gamma^+;\\
(\romannumeral2)\; & \text{$f_*$ approaches zero on the boundary of $\Gamma^+$}.
\end{split}
\end{equation*}
Then the flow (1.1) has a unique smooth and uniformly convex solution $M_t$ for all time $t>0$. For each $t\in[0,\infty)$, $X(\cdot,t)$ is a prarmetrization of a smooth, closed, uniformly convex hypersurface $M_t$ in $\mR^{n+1}$ by $X(\cdot,t)$: $\mS^n\to \mR^{n+1}$. After a proper rescaling $X\to \phi^{-1}(t)X$, where
\begin{equation}
\begin{cases}
\phi(t)=e^{\gamma t} &\text{ if }\alpha=1-\beta,\\
\phi(t)=(1+(1-\beta-\alpha)\gamma t)^{\frac{1}{1-\beta-\alpha}} &\text{ if }\alpha\not=1-\beta,
\end{cases}
\end{equation}
and $$\gamma=f^\beta(1,...,1).$$The hypersurface $\widetilde M_t=\phi^{-1}M_t$ converges exponentially to a round sphere centered at the origin in the $C^\infty$-topology.
\end{theorem}
The flow(1.1) can be described by a ODE of the support function if $\beta=0$. So we donot state that result in here.

We denote by $u=<X,\nu>$ the support function of $M_t$ at $X$. Let $\Phi=u^\a,G=f^\beta$, then we rewrite the equation (1.1) in the following form
$$\frac{\p X}{\p t}=u^\alpha f^\beta(x,t) \nu(x,t)=\Phi G\nu.$$

Let us make some remarks about our conditions. Condition (1.3) ensures that this equation is parabolic. Condition (1.3),  (\romannumeral1) and (\romannumeral2) are used in our proof. We believe that conditions (\romannumeral1) and (\romannumeral2) are superfluous but we have not been able to avoid these.

We give some examples of functions $f$ satisfying the required hypotheses. For any integer $m$ such that $1\leq m\leq n$, the $mth$ elementary symmetric function $\sigma_m$ is defined by $\sigma_m(\lambda_1,...,\lambda_n)=(\sum_{1\leq i_1<\cdot\cdot\cdot<i_m\leq n}\lambda_{i_1}\cdot\cdot\cdot\lambda_{i_m})^\frac{1}{m}$. Then $\sigma_m$ is smooth, positive, symmetric function  and homogenous of degree one on the positive cone. It is easily checked that (1.3) and Condition (\romannumeral2) of Theorem 1.1  holds for $\sigma_m$. $\sigma_m$ satisfies the condition (\romannumeral1) by \cite{HGC}.

The second example is $f=(\sum_{i=1}^n\l_i^k)^\frac{1}{k}$ for $k>0$. Then $f_*(\lambda_1,...,\lambda_n)=(\sum_{i=1}^n(\frac{1}{\l_i})^k)^{-\frac{1}{k}}$ and $f$ are smooth, positive, symmetric functions on the positive cone, and both are homogenous of degree one. It is easily checked that (1.3) holds for $f$ and Condition (\romannumeral1), (\romannumeral2) hold for $f_*$.

More examples can be constructed as follows:

If $f_1,\cdots,f_k$ satisfy our conditions, then $f=\prod_{i=1}^kf_i^{\alpha_i}$ also satisfies our conditions, where $\alpha_i\ge 0$ and $\sum_{i=1}^k\alpha_i=1$.
 More example can be seen in \cite{B3,B4}.

The study of the asymptotic behavior of the flow (1.1) is equivalent to the long time behaviour of the normalised flow. Let $\widetilde X(\cdot,\tau)=\phi^{-1}(t)X(\cdot,t)$, where
\begin{equation}
\tau=\begin{cases}
t &\text{  if }\alpha=1-\beta,\\
\frac{log((1-\alpha-\beta)\gamma t+1)}{(1-\alpha-\beta)\gamma} &\text{  if }\alpha\not=1-\beta.
\end{cases}
\end{equation}
Then $\widetilde X(\cdot,\tau)$ satisfies the following normalized flow
\begin{equation}
\begin{cases}
\frac{\partial X}{\p t}(x,t)=u^\alpha f^\beta(x,t)\nu-\g X,\\
X(\cdot,0)=X_0.
\end{cases}
\end{equation}
For convenience we still use $t$ instead of $\tau$ to denote the time variable and omit the ``tilde'' if no confusions arise. In order to prove Theorem 1.1, we shall establish the a priori estimates for the normalized flow (1.6), and show that if $X(\cdot,t)$ solves (1.6), then $u$ converges exponentially to a constant as $t\to\infty$.

This paper is organized as follows. In Sect.2, we recall some properties of convex hypersurfaces and show that the flow (1.6) can be reduced to a parabolic equation of the support function. In Sect.3, we establish the a priori estimates, which ensure the long time existence of the normalized flow. Finally in Sect.4 we show that the flow (1.1) converges to the unit sphere.

\section{Preliminary}
We recall some basic notations at first. Let $M$ be a \tiaojian  hypersurface in $\mR^{n+1}$. Assume that $M$ is parametrized by the inverse Gauss map $X: \mS^n\to M\subset \mR^{n+1}$ and encloses origin. The support function $u: \mS^n\to \mR^1$ of $M$ is defined by $$u(x)=\sup_{y\in M}<x,y>.$$ The supremum is attained at a point $y=X(x)$ because of convexity, $x$ is the outer normal of $M$ at $y$. Hence $u(x)=<x,X(x)>$.

Let $e_1,\cdot\cdot\cdot,e_n$ be a smooth local orthonormal frame field on $\mS^n$, and $\nabla$ be the covariant derivative with respect to the standard metric $e_{ij}$ on $\mS^n$. Denote by $g_{ij}, g^{ij}, h_{ij}$ the induced metric, the inverse of the induced metric, and the second fundamental form of $M$, respectively. Then the second fundamental form of $M$ is given by
$$h_{ij}=\nabla_i\nabla_ju+ue_{ij},$$
and the proof can be seen in Urbas \cite{UJ}. We know $h_{ij}$ is symmetric and satisfies the Codazzi equation $$\nabla_ih_{jk}=\nabla_jh_{ik}.$$To compute the metric $g_{ij}$ of $M$ we use the Gauss-Weingarten relations $\nabla_ix=h_{ik}g^{kl}\nabla_lX$, from which we obtain$$e_{ij}=<\nabla_ix,\nabla_jx>=h_{ik}g^{kl}h_{jm}g^{ms}<\nabla_lx,\nabla_sx>=h_{ik}h_{jl}g^{kl}.$$
Since $M$ is uniformly convex, $h_{ij}$ is invertible and the inverse is denoted by $h^{ij}$, hence $g_{ij}=h_{ik}h_{jk}$. The principal radii of curvature are the eigenvalues of the matrix
\begin{equation}
b_{ij}=h^{ik}g_{jk}=h_{ij}=\nabla_{ij}^2u+u\delta_{ij}.
\end{equation}
Let $\varphi(\cdot,t): \mS^n\to \mS^n$ be the diffeomorphism such that the unit outer normal at $X(\varphi(x,t),t)$ is $x$. Then $u(x,t)=<X(\varphi(x,t),t),x>$. It follows that
\begin{equation}
\begin{split}
\frac{\p u}{\p t}=&<\frac{\p X}{\p \varphi^i}\frac{\p \varphi^i}{\p t}+\frac{\p X}{\p t},x>\\
=&<u^\alpha f^\beta x-\g X,x>\\
=&u^\alpha f^\beta-\g u,
\end{split}
\end{equation}
since $\frac{\p X}{\p \varphi_i}$ is tangential. We see therefore that the support function satisfies the initial value problem
\begin{equation}
\begin{cases}
&\frac{\p u}{\p t}=u^\alpha F^\beta([\nabla^2u+u\uppercase\expandafter{\romannumeral1}])-\g u\; \text{ on } \mS^n\times[0,\infty),\\
&u(\cdot,0)=u_0,
\end{cases}
\end{equation}
where $\uppercase\expandafter{\romannumeral1}$ is the identity matrix, $u_0$ is the support function of $M_0$, and
\begin{equation}
F([a_{ij}])=f(\mu_1,\cdots,\mu_n),
\end{equation}
where $\mu_1,\cdots,\mu_n$ are the eigenvalues of matrix $[a_{ij}]$. It is not difficult to see that the eigenvalues of $[F^{ij}]=[\frac{\p F}{\p a_{ij}}]$ are $\frac{\p f}{\p\mu_1},\cdots,\frac{\p f}{\p\mu_n}$. Thus from (1.3) we obtain
\begin{equation}
[F^{ij}]>0 \text{ on } \Gamma^+,
\end{equation}
which yields that the equation in (2.3) is parabolic for admissible solutions.

\section{A Priori Estimates}
In this section, we establish the priori estimates and show that the normalized flow exists for long time. We first show the $C^0$-estiamte of the solution to (2.3).

\begin{lemma}
Let $u(x,t)$, $t\in[0,T)$, be a smooth, uniformly convex solution to (2.3). If $\alpha\leq1-\beta$ and $\beta>0$, then there is a positive constant $C_1$ depending only on $\alpha,\beta$ and the lower and upper bounds of $u(\cdot,0)$ such that
\begin{equation*}
\frac{1}{C_1}\leq u(\cdot,t)\leq C_1.
\end{equation*}
\end{lemma}
\begin{proof}
Let $u_{\min}(t)=\min_{x\in \mS^n}u(\cdot,t)=u(x_t,t)$. For fixed time $t$, at the point $x_t$, we have $$\nabla_iu=0 \text{ and } \nabla^2_{ij}u\geq0.$$
Note that $b_{ij}=u_{ij}+u\delta_{ij}\geq u\delta_{ij}$, we have $F^\beta(b_{ij})\geq\g u^\beta$, then
$$\frac{d}{dt}u_{\min}\geq\g u_{\min}(u_{\min}^{\alpha+\beta-1}-1).$$
Hence $u_{\min}\geq \min\{1,u_{\min}(0)\}$. Similarly, we have $u_{\max}\leq \max\{1,u_{\max}(0)\}$.
\end{proof}
\begin{lemma}
 Let $\alpha\leq1-\beta$ and $\beta>0$, and $X(\cdot,t)$ be the solution to the normalized flow (1.6) which encloses the origin for $t\in[0,T)$. Then there is a positive constant $C_2$ depending on the initial hypersurface and $\alpha,\beta$,  such that
$$\frac{1}{C_2}\leq F^\beta\leq C_2.$$
\end{lemma}
\begin{proof}
Consider the auxiliary function
$$Q=u^{\alpha-1}F^\beta.$$
Then $Q=u^{\alpha-1}G$ and $G$ is homogenous of degree $\beta$. Since
$$(u^{\alpha}G)_{ij}=(Qu)_{ij}=Q_{ij}u+Q_iu_j+Q_ju_i+Qu_{ij},$$
we get
\begin{align}
\p_tQ&=\p_t(u^{\alpha-1}G)\notag\\
&=(\alpha-1)u^{\alpha-2}\frac{\p u}{\p t}G+u^{\alpha-1}G^{ij}\frac{\p h_{ij}}{\p t}\notag\\
&=(\alpha-1)Gu^{\alpha-2}(u^\alpha G-\g u)+u^{\alpha-1}G^{ij}[(u^\alpha G-\g u)_{ij}+(u^\alpha G-\g u)\delta_{ij}]\notag\\
&=(\alpha-1)Gu^{\alpha-2}(Qu-\g u)+u^{\alpha-1}G^{ij}(Q_{ij}u+Q_iu_j+Q_ju_i+Qh_{ij}-\g h_{ij})\notag\\
&=(\a+\beta-1)Q^2-(\a+\beta-1)\g Q+u^{\alpha}G^{ij}Q_{ij}+2u^{\alpha-1}G^{ij}Q_iu_j.\notag
\end{align}
If $\a+\beta-1\leq0$, the sign of the coefficient of the highest order term $Q^2$ is negative. The sign of the coefficient of the lower order term $Q$ is positive. Applying the maximum principle we know that $\frac{1}{C_3}\leq Q\leq C_3$, where $C_3$ is a positive constant depending on the initial hypersurface and $\alpha,\beta$. Then by Lemma 3.1, we have $$\frac{1}{C_2}\leq F^\beta\leq C_2.$$
\end{proof}
\begin{lemma}
Let $\alpha\leq1-\beta$ and $\beta>0$, and $X(\cdot,t)$ be the solution to the normalized flow (1.6) which encloses the origin for $t\in[0,T)$. Then there is a positive constant $C_4$ depending on the initial hypersurface and $\alpha,\beta$,  such that
$$\vert\nabla u\vert\leq C_4.$$
\end{lemma}
\begin{proof}
Let $\omega=\log u$. Then we have
\begin{gather*}
\omega_i=\frac{u_i}{u},\\\omega_{ij}=\frac{u_{ij}}{u}-\frac{u_iu_j}{u^2},\\
h_{ij}=u_{ij}+u\delta_{ij}=e^\omega(\omega_{ij}+\omega_i\omega_j+\delta_{ij}).
\end{gather*}
It is easy to see that $\omega_{ij}$ is symmetric. Thus
$$\omega_t=\frac{u_t}{u}=(e^\omega)^{\a+\beta-1}F^\beta([\omega_{ij}+\omega_i\omega_j+\delta_{ij}])-\g.$$
Consider the auxiliary function $Q=\frac{1}{2}\vert\nabla\omega\vert^2$. At the point where $Q$ attains its spatial maximum, we have
\begin{gather*}
0=\nabla_iQ=\sum_{l}\omega_{li}\omega_{l},\\
0\geq\nabla_{ij}^2Q=\sum_{l}\omega_{li}\omega_{lj}+\sum_{l}\omega_{l}\omega_{lij},
\end{gather*}
and\begin{align*}
\p_tQ_{\max}&=\sum\omega_l\omega_{lt}\\
&=\omega_l((\a+\beta-1)\omega_l(e^\omega)^{\a+\beta-1}G+(e^\omega)^{\a+\beta-1}G^{ij}(\omega_{ijl}+\omega_{il}\omega_j+\omega_i\omega_{jl}))\\
&=2(\a+\beta-1)(e^\omega)^{\a+\beta-1}Q_{\max}G+(e^\omega)^{\a+\beta-1}G^{ij}\omega_l\omega_{ijl}.
\end{align*}
We remark that here $G^{ij}=G^{ij}([\omega_{ij}+\omega_i\omega_j+\delta_{ij}])$. By the Ricci identity,$$\nabla_l\omega_{ij}=\nabla_j\omega_{li}+\delta_{il}\omega_j-\delta_{ij}\omega_l,$$
we get
\begin{align}
\p_tQ_{\max}&=2(\a+\beta-1)(e^\omega)^{\a+\beta-1}Q_{\max}G
+(e^\omega)^{\a+\beta-1}G^{ij}(\omega_l\omega_{lij}+\omega_i\omega_j-\delta_{ij}\vert\nabla\omega\vert^2)\notag\\
&\leq2(\a+\beta-1)(e^\omega)^{\a+\beta-1}Q_{\max}G+2(e^\omega)^{\a+\beta-1}(\max_iG^{ii}-\sum_iG^{ii})Q_{\max}.
\end{align}
In terms of the positive definite of the symmetric matrix $[G^{ij}]$ and $\a+\beta-1\leq0$, we have $\p_tQ_{\max}\leq0$, thus
$$\max_{\mS^n}\frac{\vert\nabla u(\cdot,t)\vert}{u(\cdot,t)}\leq \max_{\mS^n}\frac{\vert\nabla u(\cdot,0)\vert}{u(\cdot,0)}.$$
Then it follows by Lemma 3.1 that we have $\vert\nabla u(\cdot,t)\vert\leq C_4$ for a positive constant $C_4$.
\end{proof}

The next step in our proof is the derivation of a curvature radii bound.

\begin{lemma}
Let $\alpha\leq0<\beta\leq1-\alpha$, and $X(\cdot,t)$ be a \tiaojian solution to the normalised flow (1.6) which encloses the origin for $t\in[0,T)$. Then there is a positive constant $C_5$ depending on the initial hypersurface and $\alpha,\beta$,  such that the principal curvature radii of $X(\cdot,t)$ are bounded from above and below $$\frac{1}{C_5}\leq\l_i(\cdot,t)\leq C_5.$$
\end{lemma}
\begin{proof}
First, we shall prove that $\l_i$  is bounded from below by a positive constant. The principal radii of curvatures of $M_t$ are the eigenvalues of $\{h_{il}e^{lj}\}$. To derive a positive lower bound of principal curvatures radii, it suffices to prove that the eigenvalues of $\{h^{il}e_{lj}\}$ are bounded from above. For this end, we consider the following quantity
$$W(x,t)=\max\{h^{ij}(x,t)\zeta_i\zeta_j: e^{ij}(x)\zeta_i\zeta_j=1\}.$$

Fix an arbitrary $T'\in(0,T)$ and assume that $W$ attains its maximum on $\mS^n\times[0,T']$ at $(x_0,t_0)$ with $t_0>0$ (otherwise $W$ is bounded by its initial value and we are done). We choose a local orthonormal frame $e_1,\cdots,e_n$ on $\mS^n$ such that at $X(x_0,t_0)$, $\{h_{ij}\}$ is diagonal, and assume without loss of generality that
  $$h^{11}\ge h^{22}\ge\cdots\ge h^{nn}.$$

 Take a vector $\xi=(1,0,\cdots,0)$ at $(x_0, t_0)$, and extend it to a parallel vector field in a neighborhood of $x_0$ independent of $t$, still denoted by $\xi$. Set
$$\widetilde W(x,t)=\frac{h^{kl}\xi_k\xi_l}{e^{kl}\xi_k\xi_l},$$
which is differential, and there holds
$$\widetilde W(x,t)\le \widetilde W(x_0,t_0)=W(x_0,t_0)=h^{11}(x_0,t_0).$$

Note that at the point $(x_0, t_0)$
\begin{gather*}
\p_t\widetilde W=\p_th^{11}, \; \nabla_i\widetilde W=\nabla_i h^{11},\; {\mbox{ and }}\; \nabla^2_{ij}\widetilde W=\nabla^2_{ij}h^{11}.
\end{gather*}
This implies that $\widetilde W$ satisfies the same evolution as $h^{11}$ at $(x_0, t_0)$. Therefore we shall just apply the maximum principle to the evolution of $h^{11}$ to obtain the upper bound.

We denote the partial derivatives $\frac{\p h^{pq}}{\p h_{kl}}$ and $\frac{\p^2 h^{pq}}{\p h_{rs}\p h_{kl}}$ by $h^{pq}_{kl}$ and $h^{pq}_{kl,rs}$ respectively. From \cite{UJ}, we have
\begin{gather*}
h^{pq}_{kl}=-h^{pk}h^{ql},\\
h^{pq}_{kl,rs}=h^{pr}h^{ks}h^{ql}+h^{pk}h^{qr}h^{ls},\\
\nabla_jh^{pq}=h^{pq}_{kl}\nabla_jh_{kl},\\
\nabla^2_{ij}h^{pq}=h^{pq}_{kl}\nabla^2_{ij}h_{kl}+h^{pq}_{kl,rs}\nabla_ih_{rs}\nabla_jh_{kl}.
\end{gather*}
Recall that $\Phi=u^\a$ and $G=F^\beta$, we then differentiate the equation (2.3) at $(x_t,t)$ to obtain
\begin{gather}
\frac{\p}{\p t}\nabla_ku=\Phi_kG+\Phi G_k-\g u_k,\\
\frac{\p}{\p t}\nabla^2_{kl}u=\Phi_{kl}G+\Phi_{k}G_l+\Phi_lG_k+\Phi G_{kl}-\g u_{kl}.
\end{gather}
Using (2.3) and (3.3), we see that $h_{kl}=\nabla^2_{kl}u+\delta_{kl}u$ satisfies the equation
$$\frac{\p}{\p t}h_{kl}=\Phi_{kl}G+\Phi_{k}G_l+\Phi_lG_k+\Phi G_{kl}+\Phi G\delta_{kl}-\g h_{kl},$$
thus
\begin{align}
\frac{\p}{\p t}h^{11}&=h^{11}_{kl}\frac{\p}{\p t}h_{kl}\notag\\
&=-(h^{11})^2(\Phi_{11}G+2\Phi_1G_1)-(h^{11})^2\Phi(G^{ij,mn}\nabla_1h_{ij}\nabla_1h_{mn}+G^{ij}\nabla_1\nabla_1h_{ij})\notag\\&\quad-\Phi G(h^{11})^2+\g h^{11}.
\end{align}
Note that
\begin{align}
\nabla_k\nabla_lh^{11}&=h^{11}_{mn}\nabla_k\nabla_lh_{mn}+h^{11}_{mn,rs}\nabla_lh_{mn}\nabla_kh_{rs}\notag\\
&=-(h^{11})^2\nabla_k\nabla_lh_{11}+h^{ms}(h^{11})^2\nabla_lh_{1m}\nabla_kh_{1s}+h^{ns}(h^{11})^2\nabla_lh_{1n}\nabla_kh_{1s}\notag\\
&=-(h^{11})^2\nabla_k\nabla_lh_{11}+2(h^{11})^2h^{rs}\nabla_1h_{lr}\nabla_1h_{ks}.
\end{align}
By the Ricci identity,
\begin{equation}
\nabla_k\nabla_lh_{11}=\nabla_1\nabla_1h_{kl}+\delta_{1k}h_{1l}-h_{kl}+\delta_{lk}h_{11}-\delta_{1l}h_{1k},
\end{equation}
and combination of (3.4), (3.5) and (3.6) gives
\begin{align*}
\frac{\p}{\p t}h^{11}=&\Phi G^{kl}\nabla_k\nabla_lh^{11}-\Phi(h^{11})^2(G+G^{kl}h_{kl})+\Phi\sum_iG^{ii}h^{11}+\g h^{11}\\&-\Phi(h^{11})^2(2G^{km}h^{nl}+G^{kl,mn})\nabla_1h_{kl}\nabla_1h_{mn}\\&-(h^{11})^2(2\nabla_1\Phi\nabla_1G+G\nabla_1\nabla_1\Phi).
\end{align*}
Since $F=G^\frac{1}{\beta}$ is homogenous of degree one, and satisfies the conditions (\romannumeral1) and (\romannumeral2) of Theorem 1.1, it follows from Urbas \cite{UJ} that,
\begin{gather}
F^{ij}h_{ij}=F,\\
(2F^{km}h^{nl}+F^{kl,mn})\nabla_1h_{kl}\nabla_1h_{mn}\geq2F^{-1}F^{kl}F^{mn}\nabla_1h_{kl}\nabla_1h_{mn}.
\end{gather}
We then have
\begin{align*}
\frac{\p}{\p t}h^{11}=&\beta\Phi F^{\beta-1}F^{kl}\nabla_k\nabla_lh^{11}-\Phi F^\beta(\beta+1)(h^{11})^2+\beta\Phi F^{\beta-1}\sum_iF^{ii}h^{11}+\g h^{11}\\
&-\Phi(h^{11})^2(2\beta F^{\beta-1}F^{km}h^{nl}+\beta(\beta-1)F^{\beta-2}F^{kl}F^{mn}\\&+\beta F^{\beta-1}F^{kl,mn})\nabla_1h_{kl}\nabla_1h_{mn}
-(h^{11})^2(2\beta F^{\beta-1}\nabla_1\Phi\nabla_1F+F^\beta\nabla_1\nabla_1\Phi)\\
\leq&\beta\Phi F^{\beta-1}F^{kl}\nabla_k\nabla_lh^{11}-(\beta+1)\Phi F^\beta(h^{11})^2+\beta\Phi F^{\beta-1}\sum_iF^{ii}h^{11}+\g h^{11}\\
&-\beta(\beta+1)\Phi F^{\beta-2}(h^{11})^2(\nabla_1F)^2-2\beta F^{\beta-1}\nabla_1\Phi\nabla_1F(h^{11})^2\\
&-F^\beta(\a(\a-1)u^{\a-2}(\nabla_1u)^2+\a u^{\a-1}(h_{11}-u))(h^{11})^2.
\end{align*}
Since $$-2\beta F^{\beta-1}\nabla_1\Phi\nabla_1F(h^{11})^2\leq\Phi\beta(\beta+1)F^\beta(\frac{\nabla_1F}{F})^2+\frac{\beta}{\beta+1}F^\beta
\frac{(\nabla_1\Phi)^2}{\Phi},$$
we have
\begin{align*}
\p_th^{11}\leq&\beta\Phi F^{\beta-1}F^{kl}\nabla_k\nabla_lh^{11}-(\beta+1)\Phi F^\beta(h^{11})^2+\beta\Phi F^{\beta-1}\sum_iF^{ii}h^{11}+\g h^{11}\\
&+\frac{\a(\beta+1-\a)}{\beta+1}F^\beta u^{\a-2}(\nabla_1u)^2(h^{11})^2-\a u^{\a-1}F^\beta h^{11}+\a u^\a F^\beta(h^{11})^2.
\end{align*}
Since $\sum_iF^{ii}\leq Fh^{11}$ and at $(x_0,t_0)$, $\nabla_ih^{11}=0,\nabla^2_{ij}h^{11}\leq0$, we have
\begin{align*}
\p_th^{11}\leq&-\Phi F^\beta(h^{11})^2+\g h^{11}+\frac{\a(\beta+1-\a)}{\beta+1}F^\beta u^{\a-2}(\nabla_1u)^2(h^{11})^2\\
&-\a u^{\a-1}F^\beta h^{11}+\a u^\a F^\beta(h^{11})^2.
\end{align*}
If $\a\leq0$ and $\beta>0$, we have
$$\p_th^{11}\leq-C_6(h^{11})^2+C_7h^{11}.$$
That is, $h^{11}\leq C_8$, where $C_8$ depends on the initial hypersurface, $\a$ and $\beta$. Thus $\frac{1}{\l_i}\leq C_8$ for $i=1,\cdots,n$. Since $f_*(\frac{1}{\l_1},\cdots,\frac{1}{\l_n})=\frac{1}{f(\l_1,\cdots,\l_n)}$ is uniformly continuous on $\overline{\Gamma}_{c}=\{\l\in\overline{\Gamma}^+\vert\l_i\geq C_{10}$ for all $i \}$, and $f_*$ is bounded from below by a positive constant. By Lemma 3.2, condition (\romannumeral1) and (\romannumeral2) of the Theorem 1.1 imply that $\l_i$ remains in a fixed compact subset of $\bar\Gamma^+$, which is independent of $t$. That is $$\frac{1}{C_5}\leq\l_i(\cdot,t)\leq C_5.$$
\end{proof}

The estimates obtained in Lemma 3.1, 3.3 and 3.4 depend on $\a$, $\beta$ and the geometry of the initial data $M_0$. They are independent of $T$. By Lemma 3.1, 3.3 and 3.4, we conclude that the equation (2.3) is uniformly parabolic. By the $C^0$ estimate (Lemma 3.1), the gradient estimate (Lemma 3.3), the $C^2$ estimate (Lemma 3.4), Cordes and Nirenberg type estimates \cite{B2,CO,LN} and the Krylov's theory \cite{KNV}, we get the H$\ddot{o}$lder continuity of $\nabla^2u$ and $u_t$. Then we can get higher order derivation estimates by the regularity theory of the uniformly parabolic equations. Hence we obtain the long time existence and $C^\infty$-smoothness of solutions for the normalized flow (1.6). The uniqueness of smooth solutions also follows from the parabolic theory. In summary, we have proved the following theorem.
\begin{theorem}
Let $M_0$ be a \tiaojian hypersurface in $\mR^{n+1}$, $n\geq2$, which encloses the origin. If $\alpha\leq0<\beta\leq1-\alpha$, the normalized flow (1.6) has a unique \tiaojian solution $M_t$ for all time $t\geq0$. Moreover, the suport function of $M_t$ satisfies the a priori estimates
$$\parallel u\parallel_{C^{k,\beta}(\mS^n\times[0,\infty))}\leq C,$$
where the constant $C>0$ depends only on $k,\a,\beta$ and the geometry of $M_0$.
\end{theorem}

\section{Proof Of Theorem 1.1}

In this section, we prove the asymptotical convergence of solutions to the normalized flow (1.6). By Theorem 3.5 it is known that the flow (1.6) exists for all time $t>0$ and remains smooth and uniformly convex, provided $M_0$ is smooth, uniformly convex and encloses the origin. In Sect.3, we have the bound of $\l_i$, so we infer that $$\max_iG^{ii}-\sum_iG^{ii}\leq-C_9.$$ In fact that $G^{ij}$ is smooth and defined in a compact set, thus the eigenvalues of $(G^{ij})$ have the lower bound. It then follows by (3.1) that $\p_tQ_{\max}\leq-C_0Q_{\max}$ for some positive constant $C_0$, where $Q=\frac{1}{2}\vert\frac{\nabla u}{u}\vert^2$. This proves
\begin{equation}
\max_{\mS^n}\frac{\vert\nabla u(\cdot,t)\vert}{u(\cdot,t)}\leq Ce^{-C_0t},\forall t>0,
\end{equation}
for both $C$ and $C_0$ are positive constants.

\begin{proof of theorem 1.1}

Case (\romannumeral1): $\a<1-\beta.$

Let $u(\cdot,t)$ be the solution to (2.3). By making a rescaling of $M_0$ if necessary, we may assume
$$a:=\min_{\mS^n}u(\cdot,0)\leq1\leq\max_{\mS^n}u(\cdot,0)=:b.$$
Let us introduce two time-dependent functions
\begin{gather*}
u_1=[1-(1-a^{-q})e^{q\g t}]^{-\frac{1}{q}},\\
u_2=[1-(1-b^{-q})e^{q\g t}]^{-\frac{1}{q}},
\end{gather*}
where $q=\a+\beta-1<0$. It is easy to check that both $u_1$ and $u_2$ satisfy Eq (2.3). By the comparison principle, $u_1(t)\leq u(\cdot,t)\leq u_2(t)$. Hence
$$(a^{-q}-1)e^{q\g t}\leq u^{-q}-1\leq(b^{-q}-1)e^{q\g t}.$$
Thus $u$ converges to $1$ exponentially. By the interpolation and the a priori estimates established in Sect.3, we see that $\parallel u(\cdot,t)-1\parallel_{C^k(\mS^n)}\to0$ exponentially. Hence $M_t$ converges to the unit sphere centered at the origin.

Case (\romannumeral2): $\a=1-\beta$.

By (4.1), we have that $\parallel\nabla u\parallel\to0$ exponentially as $t\to\infty$. Hence by the interpolation and the a priori estimates, we can get that $u$ converges exponentially to a constant in the $C^\infty$ topology as $t\to\infty$.

\end{proof of theorem 1.1}

\section{Reference}


\end{document}